\documentclass[runningheads,envcountsame,final]{svjour2}

\smartqed  

\usepackage{graphicx}
\usepackage{mathptmx}      
\usepackage{amsmath}

\spnewtheorem*{admcond}{Admissibility condition}{\bf}{\it}
\numberwithin{equation}{section}

\newcommand{\diff}{\mathrm{d}}
\newcommand{\conv}{\mathop{\mathrm{conv}}}
\newcommand\pd[3][{}]{\mathchoice
  {\displaystyle\frac{\partial^{#1}{#2}}{\partial{#3}^{#1}}}%
  {\textstyle\partial^{#1}{#2}/\partial{#3}^{#1}}%
  {\scriptstyle\partial^{#1}{#2}/\partial{#3}^{#1}}%
  {\scriptscriptstyle\partial^{#1}{#2}/\partial{#3}^{#1}}%
}

\journalname{Archive for Rational Mechanics and Analysis}

\begin{document}

\title{On dynamics of Lagrangian trajectories for~Hamilton--Jacobi
  equations%
  \thanks{We acknowledge the support of the French Ministry for Science
  and Higher Education. %
  A.S.\ gratefully acknowledges the support of the Simons-IUM
  fellowship and hospitality of Department of Mathematics, University
  of Toronto. %
  K.K.\ acknowledges support of the NSERC Discover Grant.}
}

\author{Konstantin Khanin         \and
        Andrei Sobolevski
}

\institute{K. Khanin \and A. Sobolevski \at
              Institute for Information Transmission Problems of the
              Russian Academy of Sciences\\
              \email{sobolevski@iitp.ru}
           \and
           K. Khanin \at
           Department of Mathematics, University of Toronto
           \and
           A. Sobolevski \at
           National Research University Higher School of Economics (HSE)
}

\date{\notused}

\maketitle

\begin{abstract}
  Characteristic curves of a Hamilton--Jacobi equation can be seen as
  action minimizing trajectories of fluid particles. %
  However this description is valid only for smooth solutions. %
  For nonsmooth ``viscosity'' solutions, which give rise to
  discontinuous velocity fields, this picture holds only up to the
  moment when trajectories hit a shock and cease to minimize the
  Lagrangian action. %

  In this paper we discuss two physically meaningful
    regularisation procedures, one corresponding to vanishing
    viscosity and another to weak noise limit. %
  We show that for any convex Hamiltonian, a viscous
  regularization allows to construct a nonsmooth flow that extends
  particle trajectories and determines dynamics inside the shock
  manifolds. %
  This flow consists of integral curves of a particular
  ``effective'' velocity field, which is uniquely defined
  everywhere in the flow domain and is discontinuous on shock
  manifolds. %
  The effective velocity field arising in the weak noise limit is
  generally non-unique and different from the viscous one, but in both
  cases there is a fundamental self-consistency condition constraining
  the dynamics.

  \keywords{Hamilton-Jacobi equation \and generalized characteristics
    \and  Lagrangian action \and vanishing viscosity \and weak noise}
  \subclass{35F21 \and 49L25 \and 76N10}
\end{abstract}

\section{Introduction}
\label{sec:introduction}

\subsection{The Hamilton--Jacobi equation and viscosity solutions}

The evolutionary Hamilton--Jacobi equation,
\begin{equation}
  \label{eq:1}
  \pd\phi t + H(t,  x, \nabla\phi) = 0,
\end{equation}
appears in diverse mathematical models ranging from analytical
mechanics to combinatorics, condensed matter, turbulence, and
cosmology (see, e.g., a non-exhaustive set of references in
\cite{Bec.J:2007b}). %
In many of these applications the objects of interest are described by
singularities of solutions, which inevitably appear for generic
initial data after a finite time due to the nonlinearity
of~\eqref{eq:1}. %
Therefore one of the central issues both for theory and applications
is to understand the behaviour of the system after singularities
form. %

  In particular, behaviour of characteristics after formation of
  singularities was the subject of intensive studies in the last two
  decades. %
  The main question here is whether there exists a natural extension
  of characteristics as ``particle trajectories'' after a ``particle''
  reaches the shock manifold. %
  The problem is highly nontrivial since the velocity field is not
  well defined on the shock manifold. %
  In a series of works P.~Cannarsa and his collaborators developed the
  notion of generalized characteristics as ``integral'' curves
  satisfying a certain natural differential inclusion condition. %
  It seems however that such an approach is far too general; in
  particular, generalized characteristics are often defined not
  uniquely. %
  The examples of non-uniqueness were constructed in
  \cite{Cannarsa.P:2009} by P.~Cannarsa and Y.~Yu. %
  In fact, as we show in this paper, there are very few examples of
  uniqueness. %
  Apart from the one-dimensional case uniqueness can only happen in
  the case of quadratic Hamiltonians (see
  Section~\ref{sec:regul-with-weak}).

  Instead of considering the most general definition we propose to
  study particle trajectories corresponding to physically relevant
  regularization schemes. %
  In this paper we discuss two such types of regularization: a viscous
  regularization and a regularization by small additive noise. %
  In both cases one can construct an effective velocity field
  corresponding to the limit of vanishing regularization parameter
  (the viscosity, or the intensity of noise). %
  These two limits are essentially different apart from the two
  uniqueness cases described above.

  Central to the present paper is the idea of self-consistent
  selection of effective velocity. %
  This notion of self-consistency allows to define a unique effective
  velocity field in the case of viscous regularization. %
  We then use the notion of higher-order consistency to address the
  problem of uniqueness of particle trajectories. %
  In the case of weak noise regularization the velocity is not
  necessarily uniquely selected by the self-consistency principle.  %
  We, however, believe that entropy maximization condition will lead
  to unique dynamics in the case of general Hamiltonians~$H$ (see
  Section~\ref{sec:conclusions}). %

A useful example to be borne in mind when thinking about these
problems---and arguably the most widely known variant of
equation~\eqref{eq:1}---is the Riemann, or inviscid Burgers,
equation. %
In the physics notation (the dot $\cdot$ for inner product and
$\nabla$ for spatial gradient) this equation has the form
\begin{equation}
  \label{eq:2}
  \pd{ u}t +  u\cdot\nabla u = 0,\qquad u = \nabla\phi.
\end{equation}
The first eq.~\eqref{eq:2} corresponds to the Hamiltonian $H(t, x, 
p) = | p|^2/2$. %
This equation may in turn be considered as a limit of vanishing
viscosity of the Burgers equation
\begin{equation}
  \label{eq:3}
  \pd{ u^\mu}t +  u^\mu\cdot\nabla u^\mu
  = \mu{\nabla}^2 u^\mu,\qquad  u^\mu = \nabla\phi^\mu,
\end{equation}
so solutions of~\eqref{eq:2} can be defined as limits of smooth
solutions to~\eqref{eq:3} as the positive parameter $\mu$ goes to
zero. %

The Burgers equation is in fact very special: it can be exactly mapped
by the Cole--Hopf transformation into the linear heat equation and
therefore explicitly integrated, which in turn allows to explicitly
study the limit $\mu \downarrow 0$ \cite{Hopf.E:1950}. %
Although in the case of general convex Hamiltonian the Hopf--Cole
transformation is not available, the qualitative behaviour of
solutions to a parabolic regularization of~\eqref{eq:1}
\begin{equation}
  \label{eq:4}
  \pd{\phi^\mu}t + H(t,  x, \nabla\phi^\mu) = \mu{\nabla}^2\phi^\mu
\end{equation}
as viscosity vanishes is similar to that for the Burgers equation. %
The limit $\phi(t,  x)= \lim_{\mu \downarrow 0} \phi^\mu(t,  x)$
exists and is called the \emph{entropy} (or \emph{viscosity})
\emph{solution}. %

A theory of weak solutions for a general Hamilton--Jacobi equation,
employing the regularization by infinitesimal viscosity, exists since
the 1970s \cite{Kruzhkov.S:1975,Lions.P:1982,Crandall.M:1992}. %
In the one-dimensional setting this theory is essentially equivalent
to the earlier theory of hyperbolic conservation laws
\cite{Hopf.E:1950,Lax.P:1954,Lax.P:1957,Oleuinik.O:1954}. %
The theory of weak solutions for the Hamilton--Jacobi equation is
closely related to calculus of variations, and introduction of
diffusion corresponds to stochastic control arguments
\cite{Fleming.W:2005}. %
The viewpoint of the present paper is somewhat complementary: the
Hamilton--Jacobi equation is considered as a fluid dynamics model, and
the main goal is to construct a flow of ``fluid particles'' inside the
shocks of a weak solution. %
However it is convenient to start with the Lax--Oleinik variational
pronciple, which provides a purely variational construction of the
viscosity solution. %
Remarkably this construction does not use any explicit viscous
regularization. %

\subsection{The variational construction of viscosity solutions and shocks}

Assume that the Hamiltonian function $H(t,  x,  p)$ is
smooth and strictly convex in the momentum variable $ p$, i.e., is
such that for all $(t,  x)$ the graph of $H(t,  x,  p)$ as a
function of $ p$ lies above any tangent plane and contains no
straight segments. %
This implies that the formula $ v = \nabla_{ p} H(t,  x, 
p)$ establishes a one-to-one correspondence between values of
velocity~$ v$ and momentum~$ p$. %
Moreover, the Lagrangian function
\begin{equation}
  \label{eq:5}
  L(t,  x,  v) = \max_{ p}\,
  [ p\cdot v - H(t,  x,  p)],
\end{equation}
under the above hypotheses is smooth and strictly convex in~$ v$. %
(Note that $L$ may be not finite everywhere: e.g., the relativistic
Hamiltonian $H(t,  x,  p) = \sqrt{1 + | p|^2}$ corresponds to
the Lagrangian $L(t,  x,  v)$ that is defined for $| v| \le
1$ as $-\sqrt{1 - | v|^2}$ and takes value $+\infty$ elsewhere. %
This does not happen if in addition one assumes that the Hamiltonian
$H$ grows superlinearly in $| p|$.) %

The relation between the Lagrangian and the Hamiltonian is symmetric:
they are Legendre--Fenchel conjugate~\eqref{eq:5} to one another. %
This relation can also be expressed in the form of the Young
inequality:
\begin{equation}
  \label{eq:6}
  L(t,  x,  v) + H(t,  x,  p) \ge  v\cdot p,
\end{equation}
which holds for all $ v$ and~$ p$ and turns into equality
whenever $ v = \nabla_{ p}H(t,  x,  p)$ or equivalently
$ p = \nabla_{ v} L(t,  x,  v)$. %
The two maps $ p \mapsto \nabla_{ p}H(t,  x,  p)$ and~$
v \mapsto \nabla_{ v} L(t,  x,  v)$ are thus inverse to each
other; we will call them the \emph{Legendre transforms} at $(t, 
x)$ of $ p$ and of~$ v$. %
(Usually the term ``Legendre transform'' refers to the relation
between the conjugate functions $H$ and~$L$; here we follow the usage
that is adopted by A.~Fathi in his works on weak KAM theory
\cite{Fathi.A:2016} and is more convenient in the present context.) %

Note that if $H(t,  x,  p) = | p|^2/2$, then $L(t,  x, 
v) = | v|^2/2$ and the Legendre transform reduces to the identity
$ v =  p$, blurring the distinction between velocities and
momenta. %
This is another special feature of the (inviscid) Burgers equations. %

Now assume that $\phi(t,  x)$ is a strong solution of the inviscid
equation~\eqref{eq:1}, i.e., a $C^1$ function that satisfies the
equation in the classical sense. %
For an arbitrary differentiable trajectory $\gamma(t)$ the full time
derivative of~$\phi$ along $\gamma$ is given by
\begin{equation}
  \label{eq:7}
  \frac{\diff\phi(t, \gamma)}{\diff t}
  = \pd\phi t + \dot{\gamma}\cdot\nabla\phi
  = \dot{\gamma}\cdot\nabla\phi- H(t, \gamma, \nabla\phi)
  \le L(t, \gamma, \dot{\gamma}),
\end{equation}
where at the last step the Young inequality~\eqref{eq:6} is used. %
This implies a bound for the mechanical action corresponding to the
trajectory $\gamma$:
\begin{equation}
  \label{eq:8}
  \phi(t_2, \gamma(t_2)) \le
  \phi(t_1, \gamma(t_1)) +
  \int_{t_1}^{t_2}
  L(s, \gamma(s), \dot{\gamma}(s))\, \diff s.
\end{equation}
Equality in~\eqref{eq:7} is only achieved if $\dot{\gamma}$ is the
Legendre transform of $\nabla\phi$ at every point $(t, \gamma(t))$:
\begin{equation}
  \label{eq:9}
  \dot{\gamma}(t) = \nabla_{ p}
  H(t, \gamma, \nabla\phi(t, \gamma)).
\end{equation}
Therefore the bound~\eqref{eq:8} is achieved for trajectories
satisfying Hamilton's canonical equations, with momentum given for the
trajectory~$\gamma$ by $ p_{\gamma}(t) := \nabla\phi(t,
\gamma(t))$. %
(The second canonical equation, $\dot{ p} = -\nabla_{ x} H$,
follows from \eqref{eq:1} and~\eqref{eq:9} for a $C^2$ solution~$\phi$
because
\begin{equation} 
  \label{eq:10}
  \dot{ p}_{\gamma}(t) = \pd{\nabla\phi}t
  + \dot{\gamma}\cdot(\nabla\otimes\nabla\phi)
  = -\nabla_{ x}H(t, \gamma, \nabla\phi)
  - \nabla_{ p}H\cdot (\nabla\otimes\nabla \phi)
  + \dot{\gamma}\cdot(\nabla\otimes\nabla\phi),
\end{equation}
where the last two terms cancel.) %

This is a manifestation of the variational \emph{principle of the
  least action}: Hamiltonian trajectories $(\gamma(t), 
p_{\gamma}(t))$ are (locally) action minimizing. %
In particular, if the initial condition
\begin{equation}
  \label{eq:11}
  \phi(t = 0,  y) = \phi_0( y),
\end{equation}
is a fixed smooth function, the identity
\begin{equation}
  \label{eq:12}
  \phi(t,  x) = \phi_0(\gamma(0))
  + \int_0^t L(s, \gamma(s), \dot{\gamma}(s))\, \diff s
\end{equation}
holds for an Euler--Lagrange trajectory $\gamma$ such that 
$\gamma(t) =  x$ and $ p_{\gamma}(0) =
\nabla\phi_0(\gamma(0))$. %

However the least action principle has wider validity: in fact it can
be used to \emph{construct} the viscosity solution corresponding to
the initial data~\eqref{eq:11}:
\begin{equation}
  \label{eq:13}
  \phi(t,  x) = \min_{\gamma\colon \gamma(t) =  x}
  \Bigl(\phi_0(\gamma(0))
  + \int_0^t L(s, \gamma(s), \dot{\gamma}(s))\, \diff s\Bigr).
\end{equation}
This is the celebrated Lax--Oleinik formula (see, e.g.,
\cite{E.W:2000} or \cite{Fathi.A:2016}), which reduces a PDE problem
\eqref{eq:1},~\eqref{eq:11} to the variational problem~\eqref{eq:13}
where minimization is extended to all sufficiently smooth (in fact
absolutely continuous) curves $\gamma$ such that $\gamma(t) =
 x$. %

At those points $(t,  x)$ where the function $\phi$ defined
by~\eqref{eq:13} is smooth in~$ x$, the minimizing trajectory is
unique. %
In this case, the minimizer can be embedded in a smooth family of
minimizing trajectories whose endpoints at time $0$ and $t$ are
continuously distributed about $\gamma(0)$ and $\gamma(t) = 
x$ (a convenient reference is \cite[Section 6.4]{Cannarsa.P:2004},
although this fact is classical). %
A piece of initial data $\phi_0$ gets continuously deformed according
to~\eqref{eq:7} along this bundle of trajectories into a piece of
smooth solution $\phi$ to~\eqref{eq:1} defined in a neighbourhood
of~$ x$ at time~$t$. %
Of course the Hamilton--Jacobi equation is satisfied by $\phi$ in
strong sense at all points where it is differentiable. %

But the crucial feature of~\eqref{eq:13} is that generally there will
be points $(t,  x)$ with several minimizers $\gamma_i$ that
start at different locations $\gamma_i(0)$ and bring the same value
of action to~$ x = \gamma_i(t)$. 
Just as above, each of these Hamiltonian trajectories will be
responsible for a separate smooth ``piece'' of solution. %
Thus for locations $ x'$ close to $ x$ the function $\phi$ will
be represented as a pointwise minimum of these smooth pieces $\phi_i$:
\begin{equation}
  \label{eq:14}
  \phi(t,  x') = \min_i \phi_i(t,  x').
\end{equation}
As all $\gamma_i$ have the same terminal value of action, all the
pieces intersect at $(t,  x)$: $\phi_1(t,  x) = \phi_2(t,  x)
= \dots = \phi(t,  x)$. %
Thus the neighbourhood of $ x$ at time~$t$ is partitioned into
domains where $\phi$ coincides with each of the smooth
functions~$\phi_i$ and satisfies the Hamilton--Jacobi
equation~\eqref{eq:1} in the strong sense. %
These domains are separated by surfaces of various dimensions where
two, or possibly three or more, pieces $\phi_i$ intersect and their
pointwise minimum $\phi$ is not differentiable. %
Such surfaces are called \emph{shock manifolds} or simply
\emph{shocks}. %
Note that a function~$\phi$ defined by the Lax--Oleinik formula is
continuous everywhere, including the shocks; it is its gradient that
suffers a discontinuity. %

In general, there are infinitely many continuous functions that match
the initial condition~\eqref{eq:11} and in the complement of the shock
surfaces are differentiable and satisfy the Hamilton--Jacobi
equation~\eqref{eq:1}, just as~$\phi$ does. %
What distinguishes the function~$\phi$ defined by the variational
construction~\eqref{eq:13} from all these ``weak solutions'', and
grants it with important physical meaning, is that $\phi$ appears in
the limit of vanishing viscosity for the regularized
equation~\eqref{eq:4} with the initial condition~\eqref{eq:11} (see,
e.g., \cite{Lions.P:1982}). %
For a smooth Hamiltonian it can be proved that in a viscosity solution
minimizers can only merge with shocks but never leave them. %

Now observe that in a solution~$\phi$ given by the Lax--Ole{\u\i}nik
formula~\eqref{eq:13} a minimizer that has come to a shock cannot be
continued any longer as a minimizing trajectory: wherever it might go,
there will be other trajectories originated at $t = 0$ that will bring
smaller values of action to the same location. %
Hence for the purpose of the least action description~\eqref{eq:13},
Hamiltonian trajectories become irrelevant as soon as they hit
shocks. %
The set of trajectories which survive as minimizers until time~$t > 0$
is decreasing with $t$, but at all times it is sufficiently large to
cover the whole continuum of final positions. %

\subsection{The Lagrangian picture}

Let us now adopt an alternative ``Lagrangian''
viewpoint, assuming that
trajectories~\eqref{eq:9} are described by material
``particles'' transported by the velocity field $ u(t, x)$,
which is the Legendre transform of the momenta field $ p(t, x)=
\nabla \phi(t, x)$. %
From this new perspective it is no longer natural to accept that
particles annihilate once they reach a shock. %
Can therefore something be said about the dynamics of those particles
that got into the shock, notwithstanding the fact that their
trajectories cease to minimize the action? %
The difficulty in such an approach is related to the discontinuous
nature of the velocity field $ u$, which makes it impossible to
construct classical solutions to the transport equation
$\dot{\gamma}(t) =  u(t, \gamma)$. %

In dimension~$d = 1$ the answer to the question above is readily
available. %
Shocks at each fixed $t$ are isolated points in the $ x$ space and
as soon as a trajectory merges with one of them, it continues to move
with the shock at all later times. %
This definition gives rise to dynamics that
is related to C.~Dafermos' theory of generalized
characteristics (see \cite{Dafermos.C:2005} and references
  therein) which, in fact, can be extended
to a much more general situation of nonconvex Hamiltonians and systems
of conservation laws. %
However, in several space dimensions shock manifolds are extended
surfaces of different codimension, and dynamics of trajectories inside
shocks is by no means trivial. %

For the case of the Burgers equation~\eqref{eq:3}
dynamics inside shocks was constructed in the work of I.~Bogaevsky
\cite{Bogaevsky.I:2004,Bogaevsky.I:2006} using
the following approach. %
Consider the differential equation
\begin{equation}
  \label{eq:15}
  \dot{\gamma}^\mu(t) =  u^\mu(t, \gamma^\mu),\qquad
  \gamma^\mu(0) =  y.
\end{equation}
Since $ u^\mu$ for $\mu > 0$ is a smooth vector field, this
equation defines a family of particle trajectories that form a smooth
flow. %
The next step is to take the limit of this flow as $\mu\downarrow
0$. %
It turns out that this limit exists as a non-differentiable continuous
flow, for which the forward derivative $\dot{\gamma}(t + 0) =
\lim_{\tau\downarrow 0}[\gamma(t + \tau) - \gamma(t)]/\tau$ is
defined everywhere. %
If $\gamma(t)$ is located outside shocks, this derivative coincides
with $ u(t, \gamma(t))$. %
Otherwise the effective velocity~ $\dot{\gamma}(t + 0)$ is
determined by the extremal values of velocities $ u_i =
\nabla\phi_i$ at the shock, and there is an interesting explicit representation for
it: $\dot{\gamma}(t + 0)$ coincides with the center of the smallest
ball that contains all~$ u_i$ (fig.~\ref{fig:bogaevsky}). %
The limiting flow turns out to be \emph{coalescing} (and therefore not
time-reversible): once any two trajectories
intersect, they stay together for all later times.

\begin{figure}
  \centering
  \begin{tabular}{ccc}
    \includegraphics{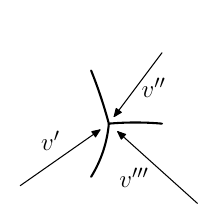} &
    \includegraphics{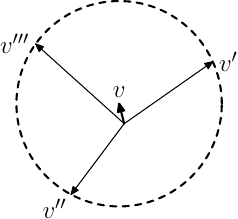} &
    \includegraphics{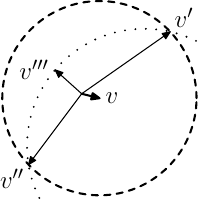} \\
    (a) & (b) & (c)
  \end{tabular}
  \caption{\small Bogaevsky's construction
    \protect\cite{Bogaevsky.I:2004,Bogaevsky.I:2006} of the effective
    velocity $v$ at a triple shock point in dimension $d = 2$: (a) the
    local structure of the flow ($v'$, $v''$, and $v'''$ are the
    limiting values of velocity when the triple point is approached
    from three different domains of smooth flow); (b) the effective
    velocity $v$ is the center of the smallest circle containing the
    three limiting velocities; (c) the smallest circle (dashed) is not
    necessarily the circumscribed one (dotted), so the effective
    velocity may be determined by a proper subset of limiting
    velocities (here, $v'$ and~$v''$).}
  \label{fig:bogaevsky}
\end{figure}

Moreover, it turns out that pieces of the shock manifold may be
classified into \emph{restraining} and \emph{nonrestraining} depending
on whether trajectories stay on them or leave them along pieces of
shock manifold of lower codimension%
\footnote{In the example of fig.~\ref{fig:bogaevsky}, case (c)
  corresponds to a nonrestraining triple point in $d = 2$, which
  trajectories leave through the shock line that divides domains of
  smooth flow with limiting velocities $v'$ and~$v''$; see fig.~3 from
  \cite{Bogaevsky.I:2004} and the discussion therein.}. %
Shocks of codimension one are always restraining; in particular, such
are all shocks in the one-dimensional case. %
Interestingly, this classification, introduced for the first time by
Bogaevsky in~\cite{Bogaevsky.I:2004} (``acute'' and ``obtuse''
superdifferentials of~$\phi$) seems to have been overlooked by
physicists despite its clear physical
significance (S.~Gurbatov and S.~Shandarin, 2005, private
communication). %

The proofs of all these facts
in~\cite{Bogaevsky.I:2004,Bogaevsky.I:2006} were based on specific
properties of the quadratic Hamiltonian and cannot be extended to the
general setting of a convex Hamiltonian. %
In this work we follow a different approach to the
vanishing viscosity limit in the general setting, employing
  the idea of self-consistency and based on the variational
  representation proposed in our earlier work
  \cite[Subsection~4.2]{Bec.J:2007b} (see also
  \cite[Section~3]{Bogaevsky.I:2006}). %
This approach, which employs the fundamental uniqueness of
the possible limiting behaviour of~$\gamma^\mu$, leads to results on
existence, uniqueness, and explicit representation of the limit
velocities. %

It should be remarked that equation \eqref{eq:9}, which relates the
velocity $\dot{\gamma}$ of a trajectory to the gradient~$\nabla\phi(t,
\gamma)$ of the solution, can be seen as defining a generalization of
the gradient flow of the function~$\phi$
\cite{Fathi.A:2016,Bogaevsky.I:2006}. %
Such a flow coincides with the conventional gradient flow when $H( p)
= | p|^2/2$ and $\phi$ is smooth. %
The case of a concave (or semiconcave) nonsmooth~$\phi$ can be handled
using the differential inequality that goes back to the work of
H.~Br{\'e}zis \cite{Brezis.H:1973}. %
A similar approach was also used by P.~Cannarsa and C.~Sinestrari in
the context of propagation of singularities for the eikonal equation
with quadratic Hamiltonian \cite[Lemma~5.6.2]{Cannarsa.P:2004}. %
Later, Cannarsa and Y.~Yu \cite{Cannarsa.P:2009} proposed an
  approach to the case of a general convex Hamiltonian that leads to
  selection of the same effective velocity as in
  \cite{Bec.J:2007b,Bogaevsky.I:2006}. %
  The proofs in \cite{Cannarsa.P:2009} are based on a mollification
  argument.  %
The classification of singluarities into ``restraining'' and
``nonrestraining'' seems however to have been unknown before
Bogaevsky's work~\cite{Bogaevsky.I:2004} even in the quadratic case. %

\subsection{Outline of the paper}

In Section~\ref{sec:soft} we develop a local theory for Lagrangian
particles in a gradient flow defined by a viscosity solution~$\phi$. %
Here we introduce the notions of \emph{admissible velocity} and
\emph{admissible momentum} at a shock, which are central to our
approach. %
The admissible velocity at each point turns out to be the unique
solution to a particular convex minimization problem
\cite{Bogaevsky.I:2006,Bec.J:2007b} (see also
  \cite{Cannarsa.P:2009}), which extends the construction of the
center of the smallest Euclidean ball (cf fig.~\ref{fig:bogaevsky}) to
the general convex case. %

In Section~\ref{sec:vanish-visc-limit} the limit of a flow regularized
with small viscosity is shown to be tangent to the field of admissible
velocities. %
This establishes an existence theorem for integral curves of this
field. %
We further discuss the issue of uniqueness of the limiting
  trajectories and propose a formal perturbative approach that allows
to determine the higher time derivatives of limiting trajectories. %

In Section~\ref{sec:regul-with-weak}, a different approach
  to regularization is discussed, which is in a sense dual to
  regularization with vanishing viscosity: regularization with weak
  noise. %
  Although the two approaches are parallel and in particular both
  feature a self-consistency condition that selects ``good''
  velocities at singular points, it turns out that in the latter
  approach the self-consistent velocity may fail to be unique in
  dimensions greater than~$2$. %
  We also show that the weak noise regularisation generally
  corresponds to a different effective velocity field.

The concluding section contains a discussion and a list of
  several open problems.

\section{Viscosity solutions and admissible gradient vector fields}
\label{sec:soft}

\subsection{Superdifferentials of viscosity solutions}
\label{sec:local-struct-visc}

Let $\phi$ be a viscosity solution to the Hamilton--Jacobi
equation~\eqref{eq:1} with initial data~\eqref{eq:11}. %
We shall use the following standard facts, for which we refer the
reader again to the recent and useful exposition in
\cite[Section 6.4]{Cannarsa.P:2004}, although many of these facts
date from 50 and more years ago: %
(i) the function $\phi$ is locally uniformly semiconcave in $(t, 
x)$ variables; (ii) if there is a single minimizer coming to~$(t, 
x)$, then $\phi$ is differentiable at this point, $C^2$ smooth in some
its neighbourhood, and
\begin{align}
  \label{eq:16}
  \phi(t + \tau,  x + \xi)
  & =  \phi(t,  x) + \pd\phi t\,\tau
  + \nabla\phi\cdot\xi + o(|\tau| + |\xi|) \\
  \label{eq:17}
  & =  \phi(t,  x) - H(t,  x, \nabla\phi)\, \tau
  + \nabla\phi\cdot \xi + o(|\tau| + |\xi|);
\end{align}
(iii) if $\phi$ is not differentiable at $(t,  x)$ and there is a
finite number of minimizers $\gamma_i$ such that $\gamma_i(t) =
 x$, then each of them corresponds to a different smooth
branch~$\phi_i$ of solution defined in the neighbourhood of~$(t, 
x)$. %
Then the Lax--Oleinik formula implies that
\begin{align}
  \label{eq:18}
  \phi(t + \tau,  x + \xi)
  & = \min_i \phi_i(t + \tau,  x + \xi) \\
  \label{eq:19}
  & = \phi(t,  x) + \min_i (-H_i\tau +  p_i\cdot\xi) +  o(|\tau| + |\xi|),
\end{align}
where $ p_i := \nabla\phi_i(t,  x)$ and $H_i := H(t,  x, 
p_i)$. %

In the latter case neither of expressions $-H_i \tau + 
p_i\cdot\xi$ provides a valid linear approximation to the
difference $\phi(t + \tau,  x + \xi) - \phi(t,  x)$ at all
points, but they all \emph{majorize} this difference up to a remainder
that is either linear or higher-order, depending on $\tau$
and~$\xi$. %
Evidently, so does the linear form $-H\tau +  p\cdot\xi$ for any
convex combination
\begin{equation}
  \label{eq:20}
   p = \sum_i \lambda_i  p_i,\qquad
  H = \sum_i \lambda_i H_i
\end{equation}
with $\lambda_i \ge 0$, $\sum_i \lambda_i = 1$. %
In convex analysis these convex combinations are called
\emph{supergradients} of $\phi$ at $(t,  x)$ and the whole
collecton of them, which is a convex polytope with vertices $(-H_i,
 p_i)$, is called the \emph{superdifferential} of $\phi$
\cite{Cannarsa.P:2004,Rockafellar.R:1970}. %
We use for the superdifferential the notation $\partial\phi(t, 
x)$. %

To avoid a possible misunderstanding it should be noted that, although
uniqueness of a minimizer coming to~$(t, x)$ implies differentiability
in~$x$ of a nonsmooth solution~$\phi$ to the Hamilton--Jacobi equation
at $t$ and earlier times, it does \emph{not} imply its
differentiability at any $t + \tau > t$. %
However the differentiability is recovered as
  $\tau\downarrow 0$, because the corresponding superdifferential
  shrinks to the gradient of~$\phi$ at $(t, x)$. %
Such points $(t, x)$ where the differentiability cannot be extended to
an open neighbourhood in spacetime are often called
\emph{preshocks}\label{pg:preshock} (see \cite{Bec.J:2007b}
  and references therein) and correspond to conjugate points of a
corresponding variational problem; a classification of all the
possible combinations of shocks and preshocks in dimensions $d = 2$
and~$d = 3$ is provided in \cite{Bogaevsky.I:2002}. %
Note that at a preshock the linearization of $\phi_i$ does not fully
anticipate the shocks at times $t + \tau$ for any $\tau > 0$. %

Under a viscous regularization $\phi^\mu$ of the solution~$\phi$, a
shock point $(t,  x)$ is ``smeared'' over a small area where
$(\pd{\phi^\mu} t, \nabla\phi^\mu)$ takes on all values from the relative
interior of $\partial\phi(t,  x)$. %
Thus, intuitively, $\partial\phi(t,  x)$ is a set of values taken
by the spacetime gradient of the nonsmooth function~$\phi$ in an
infinitesimal neighbourhood of~$(t,  x)$. %

Completing the gradient field with superdifferentials at points where
$\phi$ is not smooth recovers, in a weaker sense, the continuity of
the map $(t,  x) \mapsto \partial\phi(t,  x)$. %
Indeed, suppose $(t_n,  x_n)$ converges to $(t,  x)$ and the
sequence $(-H_n,  p_n) \in \partial\phi(t_n,  x_n)$ has a limit
point $(-H,  p)$. %
By definition of superdifferential,
\begin{equation}
  \label{eq:22}
  \phi(t_n + \tau,  x_n + \xi) - \phi(t_n,  x_n)
  \le -H_n\, \tau +  p_n\cdot\xi +  o(|\tau| + |\xi|);
\end{equation}
passing here to the limit and using continuity of~$\phi$, we see that
$(-H,  p)\in \partial\phi(t,  x)$. %
Therefore the superdifferential $\partial\phi(t,  x)$ contains all
the limit points of superdifferentials $\partial\phi(t_n,  x_n)$ as
$(t_n,  x_n)$ converges to $(t,  x)$. %

To make this continuity argument rigorous, some control is
needed over the remainder term in~\eqref{eq:22}. %
This is easy for convex or concave functions
\cite{Rockafellar.R:1970}, for which such inequalities hold without
remainders. %
A wider function class, which contains viscosity solutions of
Hamilton--Jacobi equations and in which such control is still
possible, is formed by semiconvex or semiconcave functions
\cite{Cannarsa.P:2004}. %
We refer a reader interested in proofs of this and other convex
analytic results used in this paper to monographs
\cite{Cannarsa.P:2004,Rockafellar.R:1970}. %

\subsection{Admissible velocities and admissible momenta}
\label{sec:admissibility}

This section will describe a procedure that gives a unique possible
velocity and momentum at each point $(t,  x)$. %
The construction is based solely on the convexity of Hamiltonian in
the momentum variable. %

If $(t,  x)$ is a regular point, i.e., not a point of shock, then
the velocity $ u(t,  x)$ and the momentum $ p(t,  x)$ are
naturally defined. %
Suppose now that $(t,  x)$ is a point of shock formed by
intersection of smooth branches $\phi_i$, $i \in \mathcal I$. %
For a particle starting from a shock point $(t, x)$ its possible
velocity $ v$ must correspond to one of the the ``available''
momenta, i.e., to a momentum $ p$ that belongs to the convex hull
of the momenta $ p_i=\nabla \phi_i(t,  x)$, $i\in\mathcal
I$, or
equivalently to the $ p\/$-projection of the superdifferential
$\partial\phi(t,  x)$. %

In fact more can be said. %
For an infinitesimal positive $\tau$, when the particle has already
left its original location with
velocity~$v$, not all branches $\phi_i$ will be relevant for the
solution $\phi$ at a point $(t+\tau, x+ v\tau)$ as $\tau \downarrow
0$, but only those that contribute to the (linear approximation of the
solution), i.e., the minimum in $\min_{i\in\mathcal I}(-H_i + p_i\cdot
v)$ (cf.~\eqref{eq:19} with $\xi = v\tau$). %
The branches not contributing to the minimum can be discarded. %

Denote the set of relevant indices
\begin{equation}
  \label{eq:23}
  I( v) :=
  \{j\in\mathcal I \colon -H_j + p_j\cdot v
  = \min_{i\in\mathcal I}(-H_i + p_i\cdot v)\};
\end{equation}
it is nonempty because the minimum is attained due to convexity
of~$H(t,  x, \cdot)$. %
We can now postulate that any possible velocity $ v$ of a
Lagrangian particle inside a shock satisfies the following
condition. %

\begin{admcond}
  A velocity $ v^*$ is said to be \emph{admissible} at $(t, x)$ if
  the corresponding momentum $ p^* = \nabla_{ v} L(t,  x, 
  v^*)$ belongs to the convex hull of momenta $ p_i$ with $i \in
  I( v^*)$:
  \begin{equation}
    \label{eq:24}
     p^* \in \conv\{ p_j\colon j\in I( v^*)\}.
  \end{equation}
  This value of momentum $p^*$ is also called \emph{admissible} at
  $(t, x)$. %
\end{admcond}

Observe that since the index set $I( v)$ depends on~$ v$, this
condition can be viewed as a kind of self-consistency requirement
for trajectories. %

Equivalently, one can write 
\begin{equation}
  \label{eq:39}
   v^* \in \nabla_{ p}
  H(t,  x, \conv\{ p_j\colon j\in I( v^*)\}). 
\end{equation}
Note that, in contrast with early theory of generalized
characteristics for Hamilton--Jacobi equations
\cite[Definition 5.5.1]{Cannarsa.P:2004}, there is no convex hull
taken in~\eqref{eq:39} \emph{after} the (nonlinear) map $\nabla_{ p}
H(t, x, \cdot)$ is apllied to the superdifferential of~$\phi$ at~$(t,
x)$, even though the resulting set in the velocity space is generally
non-convex. %

The above definition allows to fix the velocity $ v^*$ uniquely. %

\begin{theorem}[uniqueness of admissible velocity]
  \label{lem:uniq}
  Let $\phi$ be a viscosity solution to the Cauchy problem
  \eqref{eq:1},~\eqref{eq:11}. %
  Then at any $(t,  x)$ there exists a unique admissible velocity
  $ v^* =  v^*(t,  x)$, which is the unique point of the
  global minimum for the function
  \begin{equation}
    \label{eq:25}
    \hat L( v) := L(t,  x,  v) - \min_{i\in\mathcal I} (-H_i + 
    p_i\cdot v).
  \end{equation}
\end{theorem}

\begin{proof}
  Recall that $L(t,  x,  v)$ is a strictly convex function
  of~$ v$ because of assumptions formulated
  in~\S\ref{sec:introduction}. %
  Rewriting
  \begin{equation}
    \label{eq:26}
    L_i( v) := L(t,  x,  v) + H_i -  p_i\cdot  v,\qquad
    \hat L( v)
    = \max_{i\in \mathcal I}\,L_i( v),
  \end{equation}
  we see that $\hat L( v)$ is a pointwise maximum of strictly
  convex functions and therefore is strictly convex itself. %
  Furthermore, because the Hamiltonian $H(t,  x,  p)$ is assumed
  to be finite for all~$ p$, its conjugate Lagrangian $L(t,  x,
   v)$ grows faster than any linear function as $| v|$
  increases, and thus all its level sets are bounded. %
  Therefore $\hat L( v)$ attains its minimum at a unique value of
  velocity~$ v^*$. %

  To simplify the presentation of ideas we start with an (elementary)
  proof of the theorem in the particular case when $I( v^*)$ is
  finite. %
  We show first that the point of minimum $ v^*$ satisfies the
  admissibility condition~\eqref{eq:24}. %
  Indeed,
  \begin{equation}
    \nabla_{ v}L_i( v^*)
    = \nabla_{ v}L(t,  x,  v^*) -  p_i
    =  p^*- p_i.
  \end{equation}
  Suppose that $ p^*$ does not belong to the convex hull of $
  p_j$, $j\in I( v^*)$. %
  Then there exists a vector $ h$ such that $( p^* -  p_j)
  \cdot  h < 0$ for all $j\in I( v^*)$. %
  It follows that $L_j( v^* + \epsilon h) < L_j( v^*)$ for
  all $j\in I( v^*)$ if $\epsilon > 0$ is sufficiently small. %
  Hence, $\hat L( v^* + \epsilon h) < \hat L( v^*)$ for
  sufficiently small $\epsilon$, which contradicts to our assumption
  that $ v^*$ is a point of minimum. %
  This contradiction proves that $ v^*$ is admissible. %

  To prove uniqueness we show that if $\hat{ v}$ is admissible then
  it is a (necessarily unique) point of global minimum for the
  strictly convex function $\hat L$. %
  Using the strict convexity of $L_j$, we obtain
  \begin{align}
    L_j(\hat{ v} +  h) & = L(t,  x, \hat{ v} + h) + H_j
    -  p_j\cdot(\hat{ v} + h) \\
    & > L_j(\hat{ v}) + \nabla_{ v}L(t,  x, \hat{
      v})\cdot h -  p_j\cdot h = L_j(\hat{ v}) + (\hat{
      p} -  p_j)\cdot  h,
  \end{align}
  where $\hat{ p}$ is the Legendre transform of $\hat{ v}$. %
  Since $\hat{ v}$ is admissible, $\hat{ p} = \sum_j\lambda_j
  p_j$, where all $\lambda_j\geq 0$ and $\sum_j \lambda_j = 1$. %
  Hence, $\sum_j\lambda_j(\hat{ p} -  p_j)\cdot  h =
  [(\sum_j\lambda_j)\hat{ p} - \sum_j\lambda_j p_j]\cdot h =
  [\hat{ p} - \hat{ p}]\cdot h = 0$. %
  It follows that $(\hat{ p} -  p_j)\cdot  h > 0$ for at
  least one $j\in I(\hat{ v})$. %
  This implies that $\hat L(\hat{ v} +  h)
  > \hat L(\hat{ v})$, which means that $\hat{ v}$ is a point
  of global minimum for $\hat L$. %

  In the general case of an arbitrary~$I( v^*)$, only the first
  argument, namely admissibility of the global minimum~$ v^*$, needs
  modification. %
  We shall use the following result of Clarke based on earlier work of
  Ioffe and Levin \cite[Theorem 2.8.2 and
  Corollary~1]{Clarke.F:1990}. %
  Let $\hat L( v) = \max_{i \in \mathcal I} L_i(v)$, where
  $\mathcal I$ is a compact topological space, and suppose that all
  functions $L_i(\cdot)$ are convex and Lipschitz with the same
  constant and $I( v)$ is the set of $i$'s for which the maximum is
  attained; then the subdifferential $\partial \hat L( v)$ is a
  weakly-$*$ closed convex hull of the union of $\partial L_i( v)$ for
  $i\in I( v)$. %
  (To justufy compactness of~$\mathcal I$, observe that one can use
  the values of momenta $ p_i$ instead of the abstract indices~$i$ and
  that the set of minimizers coming to~$(t, x)$ is closed and their
  momenta are bounded.) %
  Now take into account that in our case
  $\partial L_i( v) = \nabla_{ v} L(t, x, v) - p_i$ and that $ v^*$ is
  the point of minimum, i.e., that
  \begin{displaymath}
     0 \in \partial\hat L( v^*) = \conv\{\,\nabla_{ v} L(t,  x,
     v^*) -  p_i\colon i \in I( v^*)\,\} = \{\, p^* - 
    p_i\colon i \in I( v^*)\,\};
  \end{displaymath}
  this coincides with the admissibility condition $ p^* \in \conv
  \{\, p_i\colon i \in I( v^*)\,\}$ \eqref{eq:24}. %
\end{proof}

Thus the admissibility property, first formulated above in the hardly
manageable combinatorial form~\eqref{eq:24}, turns out to be the
optimality condition for a convex minimization
problem~\eqref{eq:25}. %
In particular, if $\phi$ is differentiable at $(t,  x)$, then $\hat
L( v) = L(t,  x,  v) + H(t,  x, \nabla\phi) -
\nabla\phi\cdot v$ and the minimum in~\eqref{eq:25} is achieved at
the Legendre transform of~$\nabla\phi$. %
We thus recover Hamilton's equation~\eqref{eq:9}. %

The following reformulation will clarify the connection between
admissibility and the original construction for the Burgers equation
proposed by Bogaevsky in \cite{Bogaevsky.I:2004,Bogaevsky.I:2006}. %
Let $ v_i = \nabla_{ p} H(t,  x,  p_i)$ be the velocity
corresponding to the limit momentum~$ p_i$ and observe that $
p_i = \nabla_{ v}L(t,  x,  v_i)$. %
The Legendre duality implies that $H_i = H(t,  x, 
p_i) =  p_i\cdot  v_i - L(t,  x,  v_i)$ and
therefore~\eqref{eq:26} assumes the form
\begin{equation}
  \label{eq:27}
  \hat L( v) = \max_{i\in \mathcal I}\,
  [L(t,  x,  v) - L(t,  x,  v_i)
  - \nabla_{ v}L(t,  x,  v_i)\cdot( v -  v_i)]
  = \max_{i\in \mathcal I}\, D_L^{t,  x}( v\mid  v_i).
\end{equation}
The quantity in square brackets is known as the \emph{Bregman
  divergence} $D_L^{t,  x}( v\mid  v_i)$ of vector $ v$
with respect to~$ v_i$, a non-symmetric measure of separation of
vectors with respect to the convex function~$L(t,  x, \cdot)$
\cite{Bregman.L:1967}. %
Theorem~\ref{lem:uniq} terefore means that the admissible velocity is
the center of the the smallest ``Bregman sphere'' containing all $
v_i$, $i \in \mathcal I$. %
When $L(t,  x,  v) = | v|^2/2$, the Bregman divergence
reduces to (half) the squared distance between the two vectors. %
Therefore the admissible velocity $ v^*$ exactly conicides with the
centre of smallest ball containing all~$ v_i$, and we recover the
result of~\cite{Bogaevsky.I:2004,Bogaevsky.I:2006}. %

Finally, let us discuss the ``physical'' meaning of the function $\hat
L$. %
Consider an infinitesimal movement from $(t,  x)$ with velocity
$ v$. %
It follows from the least action principle that $\phi(t,  x) + L(t,
 x,  v)\, \diff t - \phi(t + \diff t,  x +  v\,\diff t)
\geq 0$. %
It is easy to see that to the linear order in $\diff t$
\begin{equation}
  \phi(t,  x) + L(t,  x,  v)\, \diff t
  - \phi(t + \diff t,  x +  v\, \diff t)
  = \hat L( v)\, \diff t.
\end{equation}
Hence the unique admissible velocity $ v^*$ minimizes the rate of
growth of the difference in action between the true minimizers and
trajectories of particles on shocks. %
In other words, the trajectory inside a shock cannot be a minimizer
but it does its best to keep its surplus action growing as slowly as
possible.

\section{The vanishing viscosity limit}
\label{sec:vanish-visc-limit}

\subsection{Admissibility and uniqueness of limit velocities}
\label{sec:admiss-limit-veloc}

We have thus constructed a canonical vector field of
admissible velocities $ v^*(t,  x) = \nabla_{ p} H(t,  x,
 p^*(t,  x))$ that corresponds to a given viscosity solution
$\phi$ of the Cauchy problem \eqref{eq:1},~\eqref{eq:11}. %
Notice that in general this vector field is discontinuous on the shock
manifold. %

To see how the vector field of admissible velocities arises for
Lagrangian particles inside shocks, consider the vanishing viscosity
limit for a flow corresponding to the parabolic regularization
\begin{equation}
  \label{eq:28}
  \pd{\phi^\mu}t + H(t,  x, \nabla\phi^\mu) = \mu{\nabla}^2\phi^\mu, 
  \qquad \mu > 0,
\end{equation}
of the Hamilton--Jacobi equation~\eqref{eq:1}. %

For sufficiently smooth initial data $\phi_0( y) = \phi^\mu(t = 0,
 y)$ the partial differential equation~\eqref{eq:28} has a globally
defined strong solution~$\phi^\mu$, which is locally Lipschitz with a
constant independent of~$\mu$. %
Moreover, $\phi^\mu$ converges as $\mu \downarrow 0$ to the unique
viscosity solution~$\phi$ corresponding to the same initial data. %
Proofs of these facts may be found, e.g., in~\cite{Lions.P:1982},
where they are established for $\phi_0 \in C^{2, \alpha}$. %

Consider now the transport equation
\begin{equation}
  \label{eq:29}
  \dot{\gamma}^\mu(t) = \nabla_{ p} H(t, \gamma^\mu,
  \nabla\phi^\mu(t, \gamma^\mu)), \qquad
  \gamma^\mu(0) =  y.
\end{equation}
For $\mu > 0$ this equation has a unique solution
which continuously depends on the initial location~$ y$. %
Fix a point $(t_0,  x_0)$ with $t_0 > 0$ and pick trajectories
$\gamma^\mu$ for all sufficiently small $\mu > 0$ such that
$\gamma^\mu(t_0) \to  x_0$ as $\mu \downarrow 0$. %
The uniform Lipschitz property of solutions~$\phi^\mu$ implies that
the curves $\gamma^\mu$ are uniformly bounded and equicontinuous on
some interval containing~$t_0$. %
Hence there exists a curve $\bar{\gamma}$ and a sequence $\mu_i
\downarrow 0$ such that $\lim_{\mu_i \downarrow 0} \gamma^{\mu_i} =
\bar{\gamma}$ uniformly in~$t$ on that interval. %
Note that all $\gamma^{\mu_i}$ and~$\bar{\gamma}$ are also
Lipschitz with a constant independent of~$\mu$ and that
$\bar{\gamma}(t_0) =  x_0$. %

Let furthermore $\bar{ v}$ be a limit point of the ``forward
velocity'' of the curve~$\bar{\gamma}$ at $(t_0,  x_0)$, i.e.,
let for some sequence $\tau_k \downarrow 0$
\begin{equation}
  \label{eq:30}
  \bar{ v} = \lim_{\tau_k\downarrow 0} \frac 1{\tau_k}
  [\bar{\gamma}(t_0 + \tau_k) - \bar{\gamma}(t_0)].
\end{equation}
We cannot conclude \textit{a priori} that the curve $\bar{\gamma}$
or the velocity $\bar{ v}$ are uniquely defined. %
However it turns out that $\bar{ v}$ must satisfy the admissibility
condition with respect to the solution~$\phi$ and therefore it
coincides with the unique admissible velocity~$ v^*$. %

\begin{theorem}
  \label{lem:limitpoint}
  Let $\phi$ be a viscosity solution of the Hamilton--Jacobi equaiton
  \eqref{eq:1} with initial data~\eqref{eq:11}. %
  For any $(t_0, x_0)$ with $t_0 > 0$, any sequence $\mu_i \downarrow
  0$ such that the corresponding solutions of~\eqref{eq:29} converge
  to a curve $\bar\gamma$ uniformly in an interval containing~$t_0$ and
  $\bar\gamma(t_0) = x_0$, and any limit value $\bar v$ \eqref{eq:30},
  the velocity $\bar v$ is admissible at $(t_0, x_0)$, i.e., $\bar v
  \in \nabla_p H(t_0, x_0, \conv \{p_i\colon i\in I(\bar{ v})\})$.
\end{theorem}

\begin{proof}
Our general strategy in what follows is a proof by contradiction:
assume that $\bar{ v}$ is not admissible and show that it then
cannot be a limit velocity. %

We first set up some notation regarding geometry of the closed convex
set $\partial\phi(t_0,  x_0)$. %
Denote by $(s,  p)$ the space-time co-tangent coordinates, with $s$
a scalar dual to the time subspace and $ p$ a vector dual to the
$d$-dimensional configuration subspace. %
Let $\Lambda_{\bar{ v}}$ be the hyperplane supporting the convex
compact set $\partial\phi(t_0,  x_0)$ from below with the slope
corresponding to the velocity~$\bar{ v}$:
\begin{equation}
  \label{eq:31}
  \Lambda_{\bar{ v}}
  = \{(s,  p)\colon s = - p\cdot \bar{ v}
  + \min_{i\in\mathcal I}( p_i\cdot \bar{ v} - H_i)\}
\end{equation}
and let $\bar S$ be the intersection of $\Lambda_{\bar{ v}}$ and of
the superdifferential~$\partial\phi(t_0,  x_0)$, i.e., the face of
$\partial\phi(t_0,  x_0)$ spanned by vertices $(-H_i,  p_i)$
with indices in $I(\bar{ v})$ [cf~equation~\eqref{eq:23}]. %
In this geometric setting the admissibility condition for momentum
\eqref{eq:24} can be formulated in the following way: an admissible
momentum must belong to the $ p\/$-projection of the set~$\bar
S$. %
Denote also $\bar\Phi := \min_{i\in\mathcal I}( p_i\cdot \bar{
  v} - H_i)$. %

\begin{lemma}
  \label{lem:geometric}
  If $\bar{ p}$ does not belong to the $ p$-projection of~$\bar
  S$, then
  \begin{equation}
    \label{eq:32}
    M := \min_{(s,  p) \in \partial\phi(t_0,  x_0)}
    [s +  p\cdot \bar{ v} - \bar\Phi
    + ( p - \bar{ p})\cdot (\nabla_{ p} H(t_0,  x_0,  p)
    - \bar{ v})] > 0,
  \end{equation}
\end{lemma}

\begin{proof}
  It should be noted that $M$ in~\eqref{eq:32} is an auxiliary
  quantity, which plays role in the subsequent proof but has no
  geometric meaning by itself. %
  It can be seen as a sum of two parts, each of which is nonnegative
  for reasons related to the convexity of the Hamiltonian~$H$ and the
  superdifferential~$\partial\phi(t_0,  x_0)$, and which cannot
  simultaneously vanish. %

  Denote $\bar s = -\bar{ p}\cdot\bar{ v} + \bar\Phi$ and
  observe that the point $(\bar s, \bar{ p})$ cannot belong to
  $\partial\phi(t_0,  x_0)$ because $(\bar s, \bar{
    p})\in\Lambda_{\bar{ v}}$ but the $ p$-projection of the
  face $\bar S = \Lambda_{\bar{ v}} \cap \partial\phi(t_0, 
  x_0)$ does not contain $\bar{ p}$. %

  Monotonicity of the gradient $\nabla_{ p}H(t_0,  x_0,  p)$
  of the convex function~$H$ implies that for $ p \neq \bar{ p}$
  \begin{equation}
    \label{eq:33}
    ( p - \bar{ p})\cdot
    (\nabla_{{ p}} H(t_0,  x_0,  p) - \bar{ v}) > 0.
  \end{equation}
  Indeed, from the strict convexity of $H(t_0,  x_0, \cdot)$ in
  momentum it follows that
  \begin{displaymath}
    \begin{gathered}
      H(t_0,  x_0,  p) > H(t_0,  x_0, \bar{ p}) + ( p -
      \bar{ p})\cdot
      \nabla_{ p} H(t_0,  x_0, \bar{ p}),\\
      H(t_0,  x_0, \bar{ p}) > H(t_0,  x_0,  p) +
      (\bar{ p} -  p)\cdot \nabla_{ p} H(t_0,  x_0,  p)
    \end{gathered}
  \end{displaymath}
  whenever $ p \neq \bar{ p}$ and in particular when $(s,  p)
  \in \bar S$. %
  Adding these two inequalities and taking into account that
  $\nabla_{ p} H(t_0,  x_0, \bar{ p}) = \bar{ v}$, we
  get~\eqref{eq:33}. %

  Furthermore, as $\Lambda_{\bar{ v}}$ supports
  $\partial\phi(t_0,  x_0)$ from below, for all $(s,  p)
  \in \partial\phi(t_0,  x_0)$ we have
  \begin{equation}
    \label{eq:34}
    s +  p\cdot \bar{ v} - \bar\Phi \ge 0
  \end{equation}
  with equality only when $(s,  p) \in \Lambda_{\bar{ v}}$. %
  Thus the function of $(s,  p)$ in the square brackets
  in~\eqref{eq:32} is strictly positive on~$\partial\phi(t_0, 
  x_0)$. %
  Indeed, if $(s,  p)$ belongs to the face $\bar S$,
  then~\eqref{eq:33} is positive, and otherwise \eqref{eq:34} is
  positive. %
\end{proof}

In the rest of the proof it will be convenient to use a different
rearrangment of the expression in square brackets in~\eqref{eq:32}:
\begin{equation}
  \label{eq:35}
  [\text I] - [\text{II}] :=
  [s + ( p - \bar{ p})\cdot
  \nabla_{ p} H(t_0,  x_0,  p)]
  - [\bar\Phi - \bar{ p}\cdot \bar{ v}].
\end{equation}

Next we provide a precise meaning to the intuitive idea that for $(t,
 x)$ sufficiently close to $(t_0,  x_0)$ and $\mu$ sufficiently
small, the values of the function $\phi^\mu$ and its derivatives are
close to those for the linearization
\begin{equation}
  \label{eq:60}
  \phi(t_0,  x_0) + \min_{i\in\mathcal I}\,
  [ p_i\cdot ( x -  x_0) - (t - t_0)\, H_i]
\end{equation}
of the viscosity solution~$\phi$ near~$(t_0,  x_0)$. %

For $\epsilon > 0$ let $V_\epsilon$ be the $\epsilon$-neighbourhood
of~$\partial\phi(t_0,  x_0)$. %
Choose $\epsilon < M/(6 + 3|\bar{ v}|)$ so small that $(\bar s,
\bar{ p}) \notin V_\epsilon$ and
\begin{equation}
  \label{eq:36}
  \min_{(s,  p)\in V_\epsilon}([\text I] - [\text{II}]) \ge 2M/3 > 0.
\end{equation}

Using the upper semicontinuity of the superdifferential (see e.g.\
\cite[Proposition 3.3.4]{Cannarsa.P:2004} or
\cite[Corollary 24.5.1]{Rockafellar.R:1970}, where a similar result
is proved for convex functions), choose $R = R(\epsilon) > 0$ and $T =
T(\epsilon) > 0$ such that for all $(t,  x) \in \mathcal{D}_{T, R}
:= \{(t,  x)\colon 0 \le t - t_0 \le T,\ | x -  x_0| \le R\}$
the superdifferential $\partial\phi(t,  x)$ is contained in the set
$V_{\epsilon/2}$. %

Reducing $T$, $R$ if necessary and using the Lipschitz property
of~$\phi$,~$\phi^\mu$ (which implies boundedness of momenta) and
continuity of $\nabla_{ p} H(t,  x,  p)$ in $(t,  x)$
variables, we can assume in addition that for all $(t,  x) \in
\mathcal{D}_{T, R}$
\begin{equation}
  \label{eq:38}
  |( p - \bar{ p})
  \cdot \nabla_{ p} H(t,  x,  p)
  - ( p - \bar{ p})
  \cdot \nabla_{ p} H(t_0,  x_0,  p)| < \epsilon
\end{equation}

Denote $\Gamma(t) :=  x_0 + \bar{ v}(t - t_0)$. %
Reducing~$T$ once again, we can guarantee that for all~$t_0 \le t \le
t_0 + T$ both $|\Gamma(t) - \Gamma(t_0)| < qR$ and
$|\bar{\gamma}(t) - \bar{\gamma}(t_0)| < qR$ with any $0 < q <
1$ (this margin is needed because we will approximate
$\bar{\gamma}$ by $\gamma^\mu$, which must also belong
to~$\mathcal{D}_{T, R}$) and that, moreover, $\partial\phi(t,
\Gamma(t))$ is contained in $\bar S_{\epsilon/2}$, the
$\epsilon/2$-neighbourhood of~$\bar S$. %
The latter is possible because all limit points of $\partial\phi(t,
\Gamma(t))$ as $t \downarrow t_0$ belong to the face
of~$\partial\phi(t_0,  x_0)$ that corresponds to the
direction~$\bar{ v}$, i.e., to~$\bar S$.  A proof of this result,
which refines the upper semicontinutity property of superdifferentials
mentioned above, can be found e.g., in the context of convex functions
in~\cite[Theorem~24.6]{Rockafellar.R:1970}; its generalization to the
semiconcave case is evident. %

In what follows we will refer to the values of $\mu_i$ from the
sequence that determines~$\bar{\gamma}$, but will drop the
index~$i$ to simplfy the notation. %
Choose $\bar\mu = \bar\mu(\epsilon)$ sufficiently small so that the
following three conditions hold:
\begin{displaymath}
  (t, \bar{\gamma}(t)) \in \mathcal{D}_{T, R}\quad
  \text{for $\mu < \bar\mu(\epsilon)$}
\end{displaymath}
(this is indeed possible because $(t, \bar{\gamma}(t)) \in
\mathcal{D}_{T, qR}$ with $q < 1$),
\begin{gather}
  \label{eq:40}
  \Bigl(\pd{\phi^\mu}t(t,  x), \nabla\phi^\mu(t,  x)\Bigr) \in V_\epsilon 
\end{gather}
everywhere in~$\mathcal{D}_{T, R}$, and
\begin{equation}
  \label{eq:41}
  \Bigl(\pd{\phi^\mu}t(t, \Gamma(t)), \nabla\phi^\mu(t,
  \Gamma(t))\Bigr) \in \bar S_{\epsilon}
\end{equation}
for $t_0 < t < t_0 + T$. %
The latter two conditions hold because convergence of semiconcave
functions~$\phi^\mu$ to~$\phi$ implies that limit points of their
derivatives belong to $\partial\phi(t,  x) \subset V_{\epsilon/2}$
(in particular, $\partial\phi(t, \Gamma(t)) \subset \bar
S_{\epsilon/2}$ along the trajectory~$\Gamma$). %

We are now set for the concluding argument. %
Assume that $\bar{ v}$ is not an admissible velocity and therefore
the correspondent momentum $\bar{ p}$ does not belong to the $
p\/$-projection of~$\bar S$. %
We are going to show that in this case, although trajectories
$\gamma^\mu$ may occasionally pass close to the
trajectory~$\Gamma(t) =  x_0 + \bar{ v}(t - t_0)$, any
possible limiting value of velocity of the limit trajectory
$\bar{\gamma}$ as $\tau = t - t_0 \downarrow 0$ differs
from~$\bar{ v}$ by a positive constant. %
The central argument is provided by the following lemma.

\begin{lemma}
  \label{lem:computation}
  Under conditions of Lemma~\ref{lem:geometric} fix arbitrary positive
  $\tau < T/3$ and $\delta < M/[6(L + |\bar{ p}|)]$, where $L$ is
  the common spatial Lipschitz constant of~$\phi^\mu$ in
  $\mathcal{D}_{T, R}$ for $0 < \mu < \bar\mu$. %
  Define the cone $K_\delta := \{(t,  x) \in \mathcal{D}_{T,
    R}\colon | x - \Gamma(t)| < \delta(t - t_0)\}$ and suppose that
  $(t_0 + \tau, \bar{\gamma}(t_0 + \tau)) \in K_\delta$. %
  Then $(t, \gamma^\mu(t)) \notin K_\delta$ for all $\mu < \bar\mu$
  and $t$ such that $3\tau < t - t_0 < T$.
\end{lemma}

\begin{proof}
  The full time derivative of the function $(t,  x) \mapsto
  \phi^\mu(t,  x) - \bar { p}\cdot x$ along~$\gamma^\mu$
  is given by
  \begin{multline}
    \frac \diff{\diff t}[\phi^\mu(t, \gamma^\mu(t))
    - \bar{ p}\cdot\gamma^\mu(t)]
    = \pd{\phi^\mu}t(t, \gamma^\mu) + (\nabla\phi^\mu(t, \gamma^\mu)
    - \bar{ p})\cdot \dot{\gamma}^\mu \\
    = \pd{\phi^\mu}t(t, \gamma^\mu)
    + (\nabla\phi^\mu(t, \gamma^\mu) - \bar{ p})\cdot
    \nabla_{ p} H(t, \gamma^\mu, \nabla\phi^\mu(t, \gamma^\mu)) \\
    \ge \pd{\phi^\mu}t(t, \gamma^\mu)
    + (\nabla\phi^\mu(t, \gamma^\mu) - \bar{ p})\cdot
    \nabla_{ p} H(t_0,  x_0, \nabla\phi^\mu(t, \gamma^\mu))
    - \epsilon,
  \end{multline}
  where the last inequality follows from~\eqref{eq:38}. %
  Integrating this from $t_0 + \tau$ to~$t$ we get
  \begin{multline}
    \tag*{[I]}
    \phi^\mu(t, \gamma^\mu(t)) - \bar{ p}\cdot\gamma^\mu(t)
    - \phi^\mu(t_0 + \tau, \gamma^\mu(t_0 + \tau))
    + \bar{ p}\cdot\gamma^\mu(t_0 + \tau) \\
    \ge \int_{t_0 + \tau}^t
    \Bigl[\pd{\phi^\mu}t(t', \gamma^\mu)
    + (\nabla\phi^\mu(t', \gamma^\mu) - \bar{ p})\cdot
    \nabla_{ p} H(t_0,  x_0, \nabla\phi^\mu(t',
    \gamma^\mu))\Bigr]
    \, \diff t'\\
    -\epsilon(t - t_0 - \tau)
  \end{multline}
  On the other hand,
  \begin{multline}
    \frac \diff{\diff t}[\phi^\mu(t, \Gamma(t))
    - \bar{ p}\cdot\Gamma(t)]
    = \pd{\phi^\mu}t(t, \Gamma(t)) +
    (\nabla\phi^\mu(t, \Gamma(t)) - \bar{ p})\cdot\bar{ v} \\
    \le \bar\Phi - \bar{ p}\cdot\bar{ v}
    + \epsilon\,(1 + |\bar{ v}|),
  \end{multline}
  where we took into account~\eqref{eq:41} and the fact that $s + 
  p\cdot \bar{ v} = \bar\Phi$ for all $(s,  p) \in \bar S$
  (cf.~\eqref{eq:34}). %
  It follows that
  \begin{multline}
    \tag*{[II]}
    \phi^\mu(t, \Gamma(t)) - \bar{ p}\cdot\Gamma(t)
    - \phi^\mu(t_0 + \tau, \Gamma(t_0 + \tau))
    + \bar{ p}\cdot\Gamma(t_0 + \tau) \\
    \le \int_{t_0 + \tau}^t (\bar\Phi
    - \bar{ p}\cdot\bar{ v})\, \diff t'
    + \epsilon\,(1 + |\bar{ v}|)(t - t_0 - \tau).
  \end{multline}

  Subtracting [II] from [I],
  using~\eqref{eq:35},~\eqref{eq:36},~\eqref{eq:40} and observing that
  the Lipschitz property of $\phi^\mu$ (and correspondingly that of
  $ x\mapsto \phi^\mu(t,  x) - \bar{ p}\cdot  x$, with the
  constant $L + |\bar{ p}|$) implies that
  \begin{multline}
    \bigl|\phi^\mu(t_0 + \tau, \Gamma(t_0 + \tau))
    - \bar{ p}\cdot\Gamma(t_0 + \tau) \\
    - \phi^\mu(t_0 + \tau, \gamma^\mu(t_0 + \tau))
    + \bar{ p}\cdot\gamma^\mu(t_0 + \tau)\bigr| 
    \le (L + |\bar{ p}|)\delta\tau,
  \end{multline}
  we get
  \begin{multline}
    \label{eq:42}
    \phi^\mu(t, \gamma^\mu(t)) - \bar{ p}\cdot\gamma^\mu(t)
    - \phi^\mu(t, \Gamma(t)) + \bar{ p}\cdot\Gamma(t) \\
    \ge [\tfrac 23 M - \epsilon(2 + |\bar{ v}|)](t - t_0 - \tau)
    - (L + |\bar{ p}|)\delta\tau.
  \end{multline}
  Using again the Lipschitz property and the inequality $\epsilon \le
  M/(6 + 3|\bar{ v}|)$, we get
  \begin{equation}
    \label{eq:43}
    |\gamma^\mu(t) - \Gamma(t)|
    \ge \frac {\frac 23 M - \epsilon(2 + |\bar{ v}|)}
    {L + |\bar p|}(t - t_0 - \tau)
    - \delta\tau
    \ge \frac M{3(L + |\bar p|)} (t - t_0 - \tau) - \delta\tau.
  \end{equation}
  Since $\delta < M/[6(L + | p|)]$, this means that
  $\gamma^\mu(t)$ stays outside $K_\delta$ for $t - t_0 > 3\tau$. %
\end{proof}

We can now conclude the proof. %
Suppose that there is a sequence $t_i \downarrow t_0$ such that $(t_i,
\bar{\gamma}(t_i)) \in K_\delta$. %
Then for all sufficiently small~$\mu$ Lemma~\ref{lem:computation}
implies that $(t,\allowbreak \gamma^\mu(t)) \notin K_\delta$ when $t > t_0 +
3(t_i - t_0)$ for all~$i$, which means in turn that $(t,
\bar{\gamma}(t))$ also cannot belong to~$K_\delta$ for such~$t$. %
As $t_i \downarrow t_0$, the trajectory $\bar{\gamma}$ has to stay
outside~$K_\delta$ for all $t_0 < t < t_0 + T$, a contradiction with
what has been assumed. %
This proves Theorem~\ref{lem:limitpoint}. %
\end{proof}

A somewhat simpler proof of a similar statement can be found
  in \cite[Theorem~3.2]{Cannarsa.P:2009}; we presented the proof
  above to make our presentation self-contained. %
  Note that the presented proof also stresses the relevance of
  self-consistency condition.

\subsection{The uniqueness problem for limit trajectories }
\label{sec:uniq-probl-limit}

We have seen that limit trajectories $\bar{\gamma}$ of
solutions~$\gamma^\mu$ to the transport equation~\eqref{eq:29} are
tangent in forward time to the unique discontinuous field
of admissible velocities~$ v^*$, i.e., that $\dot{\bar{\gamma}}(t + 0)
= v^*(t, x)$ for any~$t$ and for any limit trajectory $\bar{\gamma}$
passing through some~$ x = \bar{\gamma}(t)$. %
This however does not imply that limit trajectories
\emph{themselves} are unique. %

There are in fact two different uniqueness problems: that for limit
trajectories as $\mu\to 0$ for the viscous regularization
\eqref{eq:29}, 
and that for integral curves of the differential equation
\begin{equation}
  \label{eq:64}
  \dot{\gamma}(t + 0)  =  v^*(t, \gamma).
\end{equation}
Since any limit trajectory of~\eqref{eq:29} is an integral curve
of~\eqref{eq:64} according to Theorem~\ref{lem:limitpoint}, uniqueness
for integral curves would imply uniqueness for limit trajectories. %
However, it is \textit{a priori} possbile that more than one integral
curve passes through the same singular point, but the vanishing
viscosity regularization selects only one among these curves as a
limit trajectory. %

Uniqueness of limit trajectories can be established in the case when
the Hamiltonian is quadratic in the momentum variable. %
  This follows from a particular differential inequality for the
  squared separation between two close trajectories, which follows
  from semiconcavity of the solution~$\phi$:
  \begin{equation}
    \label{eq:21}
    \frac{\diff}{\diff t} \frac{|x - y|^2}2
    = (v - w)\cdot(x - y) \le \frac C2|x - y|^2
  \end{equation}
  whenever $v \in \partial\phi(t, x)$, $w \in \partial\phi(t, y)$. %
  This inequality was used in~\cite{Brezis.H:1973} to control
  the expansion of the squared distance between trajectories in terms
  of the semiconcavity constant of the solution~$\phi$.
Indeed, two limit trajectories passing through the same point $(t, x)$
cannot diverge by a finite distance in finite time, because their
viscous regularizations must stay arbitrarily close to one another
over this time interval, provided these regularizations are close
enough at time~$t$. %
Hence, as observed in \cite{Bogaevsky.I:2004,Bogaevsky.I:2006}, the
limit flow~$\bar{\gamma}$ is defined uniquely and is therefore
\emph{coalescing}: once two trajectories intersect, they stay together
at all later times. %

In a recent paper \cite{Stromberg.T:2013} T.~Str{\"o}mberg
  proposes to \emph{postulate} the existence of a function~$\Phi$ such
  that $\Phi \ge 0$, $\Phi^{-1}(0) = \{0\}$, and $(v -
  w)\cdot\nabla\Phi(x - y) \le C\Phi(x - y)$ whenever $v$ and~$w$ are
  admissible velocities at $(t, x)$ and~$(t, y)$, respectively
  \cite[Condition~(D)]{Stromberg.T:2013}. %
  This allows to essentially repeat the above argument and to
  establish uniqueness, but of course existence of such a function
  $\Phi$ is a strong restriction and general conditions for it to hold
  are not known except in one spatial dimension or when the
  Hamiltonian is quadratic. %
  Under the same circumstances uniqueness holds for generalized
  characteristics \cite{Cannarsa.P:2009}. %
  It should be noted that the paper \cite{Stromberg.T:2013} also
  features construction of limit trajectories by a different
  regularization procedure, unrelated to viscous regularization.

Observe that the differential inequality argument
  outlined above bypasses the issue of integral curves altogether. %
It is therefore interesting to consider the existence and uniqueness
issues for the differential equation \eqref{eq:64} irrespective of
viscous regularization. %

\subsection{Perturbation theory for limit trajectories}
\label{sec:second-order-pert}

Here we show how uniqueness for the differential equation
\eqref{eq:64} can be established from a formal perturbative analysis based on rather
strong regularity assumptions. %
Let the shock manifold of~$\phi$ be locally finitely generated, i.e., suppose that
at each shock point~$(t_0,  x_0)$ there is a finite number~$k$ of
minimizers connecting that point with the initial data. %
This implies that in a neighbourhood of~$(t_0,  x_0)$ the
solution~$\phi$ may be represented locally as a pointwise minimum of a
finite number~$k$ of $C^2$ smooth branches~$\phi_i$, each of which
satisfies the Hamilton--Jacobi equation classically, and that locally
the shock manifold is composed of $C^1$ smooth pieces of different
dimensions. %
(The case of \emph{preshocks}, introduced on p.~\pageref{pg:preshock},
provides an exception to the condition of smoothness and should be
considered separately.)

One can show that on each smooth piece of the shock manifold the
spacetime field~$(1,  v^*)$ determined by admissible velocities is
a Lipschitz vector field tangent to the piece. %
The usual ODE arguments then show that the flow generated by the
vector field~$ v^*$ is uniquely defined on smooth pieces of the
shock manifold, as well as in the bulk where the solution~$\phi$ is
smooth. %

In Section~\ref{sec:admissibility} it was shown that at a shock
point~$(t_0,  x_0)$ not all of the intersecting branches $\phi_i$
of solution are relevant for the integral curve~$\gamma$ at times
$t > t_0$, but only those with $i \in I( v^*)$, i.e., those that
are relevant in the first-order (linear) approximation to both the
solution~$\phi$ and the integral curve~$\gamma$. %
Denote the corresponding index set with $\mathcal I_1 := I( v^*)$. %

Uniqueness of integral curves can only fail at shock connections:
there must be at least two pieces of shock manifold that have a common
point $(t_0,  x_0)$ and share the same tangent spacetime direction
$(1,  v^*)$ but at later times carry two disjoint trajectories both
issued from~$ x_0$ at time~$t_0$ with velocity~$ v^*$. %
Note that this is not possible if $|\mathcal I_1| \le d + 1$, where
$d$ is the spatial dimension, and the velocities $ v_i$ are in
general position: indeed, in this situation removal of any branch
$\phi_i$ with $i \in \mathcal I_1$ would change the admissible
velocity~$ v^*$. %

In fact a (formal) perturbative analysis of an integral curve
$\gamma$ in higher orders of approximation reveals a nested
sequence of finite index sets $\mathcal I_1 \supseteq \mathcal I_2
\supseteq \dots$ such that $\mathcal I_k$ lists branches relevant for
the integral curve in $k$th order, and the intersection $\mathcal I =
\cap_{k \ge 1} \mathcal I_k$ is not empty (i.e., the sequence
stabilizes). %
In particular if $|\mathcal I| \le d + 1$, then the integral
curve~$\gamma$ is defined uniquely. %

In what follows we illustrate this procedure in the second order and
obtain~$\mathcal I_2$. %

Take an integral curve~$\gamma$ such that $\gamma(t_0) = 
x_0$ and assume it to be twice differentiable at~$t_0$ in ``forward''
time:
\begin{equation}
  \label{eq:55}
  \gamma(t) =  x_0 + (t - t_0) v^* + \frac{(t -
    t_0)^2}2  a + o\bigl((t - t_0)^2\bigr),
\end{equation}
where $ v^*$ is the vector of admissible velocity at~$(t_0, 
x_0)$ and $ a$ is the yet unknown acceleration of~$\gamma$
at~$t_0$. %
At times $t = t_0 + \tau$ with sufficiently small~$\tau > 0$ the point
$\gamma(t)$ lies at intersection of a possibly smaller set of
branches $\phi_i$, which all have the same value at~$(t,
\gamma(t))$. %
The first two time derivatives of this common value along $\gamma$
can be expressed as follows. %

Using the Hamilton--Jacobi equation~\eqref{eq:1} and denoting $
p_i^{\gamma}(t) = \nabla\phi_i(t, \gamma(t))$, for the first
time derivative we get
\begin{equation}
  \label{eq:56}
  \dot\phi_i(t, \gamma(t)) = \partial_t\phi_i(t, \gamma(t))
  + \dot{\gamma}(t)\cdot \nabla\phi_i(t, \gamma(t)) 
  = \dot{\gamma}(t)\cdot p_i^{\gamma}(t)
  - H(t, \gamma(t),  p_i^{\gamma}(t)).
\end{equation}
In particular 
\begin{equation}
  \label{eq:62}
  \dot\phi_i(t_0,  x_0) =  v^*\cdot p_i - H_i,
\end{equation}
where $ p_i =  p_i^{\gamma}(t_0)$ and $H_i = H(t_0,  x_0,
 p_i)$ as above. %
Using the Legendre duality (see \eqref{eq:6} and discussion
thereafter), we can modify expression~\eqref{eq:56} as follows:
\begin{multline}
  \label{eq:46}
  \dot\phi_i(t, \bar{\gamma}(t))
  = \dot{\gamma}(t)\cdot p_i^{\gamma}(t)
  - H(t, \gamma(t),  p_i^{\gamma}(t)) \\
  = (\dot{\gamma}(t) -  v_i^{\gamma}(t))\cdot
  \nabla_{ v} L(t, \gamma(t),  v_i^{\gamma}(t))
  + L(t, \gamma(t),  v_i^{\gamma}(t)),
\end{multline}
where $ v_i^{\gamma}(t) = \nabla_{ p} H(t, \gamma,
\nabla\phi_i(t, \gamma(t)))$ and $ p_i^{\gamma}(t) =
\nabla_{ v} L(t, \gamma(t),  v_i^{\gamma}(t)) = \nabla
\phi_i(t, \gamma(t))$ are values of velocity and momentum that
correspond to the gradient~$ p_i^{\gamma}(t)$ along the
curve~$\gamma$. %
Recalling the expression for Bregman divergence~\eqref{eq:27}
\begin{equation}
  \label{eq:50}
  D_L^{t,  x}( v^*\mid  v) = L(t,  x,  v^*) - L(t, 
  x,  v) - ( v^* -  v)\cdot\nabla_{ v} L(t,  x,  v),
\end{equation}
we can now express the time derivative $\dot\phi_i(t, \gamma(t))$
in the form
\begin{equation}
  \label{eq:48}
  \dot\phi_i(t, \gamma(t)) = L(t, \gamma(t), \dot{\gamma}(t))
  - D^{t, \gamma(t)}_L (\dot{\gamma}(t) \mid 
  v_i^{\gamma}(t)).
\end{equation}

Observe that the difference between $\phi_i(t, \gamma(t))$ and the
mechanical action along the curve~$\gamma(\cdot)$ decreases as
the (negative) integral over $(t_0, t)$ of the Bregman divergence
$D^{\cdot, \bar{\gamma}(\cdot)}_L (\dot{\gamma} \mid 
v_i^{\gamma})$. %
Of course subtracting the common quantity from the values of branches
$\phi_i(t, \gamma(t))$ for all~$i$ does not change the mutual order
of these values. %
We notice that the bigger is the Bregman divergence $D^{\cdot,
  \gamma(\cdot)}_L (\dot{\gamma} \mid  v_i^{\gamma})$, the
faster decreases this difference: up to the second order in~$t - t_0$,
the value $\min_i \phi_i$ will be attained at the branch or branches
for which $\dot{\gamma}(t)$ is the most distant (in the Bregman
sense) from~$ v_i^{\gamma}(t)$.

To obtain the second time derivative we differentiate the r.h.s.\
of~\eqref{eq:56} to get
\begin{multline}
  \label{eq:57}
  \ddot\phi_i(t, \gamma(t))
  = \ddot{\gamma}(t)\cdot p_i^{\gamma}(t)
  + \dot{\gamma}(t)\cdot\dot{ p}_i^{\gamma}(t)
  -  v_i^{\gamma}(t)\cdot\dot{ p}_i^{\gamma}(t) \\
  - \Bigl[\pd{}t H(t, \gamma(t),  p_i^{\gamma}(t))
  + \dot{\gamma}(t)\cdot \nabla_{ x}
  H(t, \gamma(t),  p_i^{\gamma}(t))\Bigr].
\end{multline}
It is convenient again to consider the second time derivative not
of~$\phi_i$ itself, but of the difference between~$\phi_i$ and the
mechanical action of~$\gamma$:
\begin{multline}
  \label{eq:58}
  \ddot\phi_i(t, \gamma(t))
  - \frac{\diff}{\diff t}L(t, \gamma(t),
  \dot{\gamma}(t)) \\
  = \ddot{\gamma}(t)\cdot( p_i^{\gamma}(t)
  -  p_*^{\gamma}(t))
  + (\dot{\gamma}(t) -  v_i^{\gamma}(t))\cdot\dot{ p}_i^{\gamma}(t) \\
  - \Bigl[\pd{}t H(t, \gamma(t),  p_i^{\gamma}(t))
  + \dot{\gamma}(t)\cdot \nabla_{ x}
  H(t, \gamma(t),  p_i^{\gamma}(t))\Bigr] \\
  - \Bigl[\pd{}t L(t, \gamma(t), \dot{\gamma}(t))
  + \dot{\gamma}(t)\cdot \nabla_{ x}
  L(t, \gamma(t), \dot{\gamma}(t))\Bigr],
\end{multline}
where $ p_*^{\gamma}(t) = \nabla_{ v}L(t,
\gamma(t), \dot{\gamma}(t))$ is the value of
momentum corresponding to the velocity~$\dot{\gamma}(t)$. %
In particular at time~$t_0$ we have
\begin{equation}
  \label{eq:59}
  \ddot\phi_i - \frac{\diff L}{\diff t}
  =  a\cdot( p_i -  p^*) + ( v^* -  v_i)\cdot f_i
  - ([H]_i + [L]_i),
\end{equation}
where $ p^* =  p_*^{\gamma}(t_0)$ is the usual
admissible momentum (cf.\ \eqref{eq:24}), $ v_i = 
v_i^{\gamma}(t_0)$, $ f_i = \dot{
  p}_i^{\gamma}(t_0)$, and $[H]_i$, $[L]_i$ denote values of
the two square brackets at~$t = t_0$. %

Consider now an integral curve $\gamma$ that is determined by
intersection of smooth branches $\phi_i$ for some~$i \in \mathcal
I$. %
Two conditions must hold for small $t - t_0 > 0$ along this curve:
\begin{itemize}
\item[(i)] the velocity $\dot{\gamma}(t)$ must be admissible
  at~$(t, \gamma(t))$, i.e., be the center of the ``Bregman
  sphere'' containing all~$ v_i^{\gamma}(t)$ at its boundary;
\item[(ii)] values of the remaining branches at $(t, \gamma(t))$
  must be greater than the common value of $\phi_i(t, \gamma(t))$.
\end{itemize}

Define the piecewise linear concave function
\begin{equation}
  \label{eq:61}
  F( a) = \min_{i\in \mathcal I}\,
  \Bigl( a\cdot( p_i -  p^*)
  + ( v^* -  v_i)\cdot f_i - ([H]_i + [L]_i)\Bigr).
\end{equation}
Note that the velocity $ v^*$ of the curve~$\gamma$ at
time~$t_0$ is known, and therefore the values of $ f_i$ are the
same for any integral curve~$\gamma$, so the function $F$ can be
defined without knowing the curve~$\gamma$. %
It is easy to see that the set $\mathcal I( a)$ of indices where
minimum is attained in~\eqref{eq:61} consists of precisely those
indices for which condition (ii) holds. %
This set plays the same role in the
quadratic approximation as did the set $I^*( v)$ in the linear
approximation. %

Condition (i) then becomes an admissibility condition for the
acceleration similar to~\eqref{eq:24}. %
Geometrically, the admissible acceleration $ a$ is the value at
time~$t_0$ of the rate of change of the center of the smallest Bregman
sphere containing all~$ v_i(t)$ for sufficiently small $t - t_0$;
compare this description with the fact that the velocity
$\dot{\gamma}(t)$ is given by this center itself. %
It is clear that depending on the rates $\dot{ v}_i$ (or
equivalently, the values $\dot{ p}_i =  f_i$) at time~$t_0$,
some of the velocities present at~$t_0$ may ``sink'' into the interior
of the Bregman sphere for small $\tau = t - t_0 > 0$, leaving its
surface defined by a smaller set $\{ v_i \colon i \in \mathcal
I_2\}$. %

In a similar way one can define the index sets $\mathcal J_3$,
$\mathcal J_4$, and so on. %
Notice that this decreasing sequence of index sets will stabilize,
since their intersection is nonempty. %
We conjecture that the resulting set $\mathcal J = \cap_{s\ge 1}
\mathcal J_s$ determines the smooth manifold to which the integral
curve $\gamma$ belongs and which determines it uniquely as the
integral curve of the corresponding filed of admissible velocities. %

  \section{Regularization with weak noise}
  \label{sec:regul-with-weak}

  Observe that convergence of superdifferentials makes it possible to
  use other regularization procedures for $\phi$ (e.g., convoluting it with
  a standard mollifier), giving the same limit trajectories. %
  However, one can imagine the following completely different
  regularization of the discontinuous velocity field $\nabla_p H(t, x,
  \nabla\phi(t, x))$. %
  Physically speaking, this regularization corresponds to a zero
  ``Prandtl number,'' in contrast with the mollification approach that
  corresponds to an infinite ``Prandtl number.'' %

  Consider the stochastic equation
  \begin{equation*}
    \diff\gamma^\epsilon
    = \nabla_p H(t, \gamma^\epsilon,
    \nabla\phi(t, \gamma^\epsilon))\, \diff t + \epsilon\,\diff W(t),
  \end{equation*}
  where $W$ is the standard Wiener process. %
  The corresponding stochastic flow is well defined in spite of the
  fact that $\nabla\phi$ does not exist everywhere: whenever the
  trajectory $\gamma^\epsilon$ hits shocks, the noise in the second
  term will instantaneously steer it in a random direction away from
  the singularity. %

  One can show that as $\epsilon\downarrow 0$ the stochastic flow of
  trajectories $\gamma^\epsilon$ tends to a limit flow of trajectories
  $\tilde\gamma$, which is forward differentiable just as the flow
  constructed by the viscous regularization. %
  It is easy to see that, due to the averaging, the forward velocity
  $v^\dagger(t, \tilde\gamma) := \dot{\tilde\gamma}(t + 0)$ must
  belong to the convex hull of limit velocities $v_j$, $j \in
  I(v^\dagger)$. %
  Namely,
  \begin{equation}
    \label{eq:46_1}
    v^\dagger = \sum_{j\in I(v^\dagger)} \pi_j v_j,
    \qquad \pi_j\ge 0,
    \qquad \sum_{j\in I(v^\dagger)} \pi_j=1,
  \end{equation}
  where the velocities $v_j(t, x)$ are Legendre transforms of the
  corresponding momenta $p_j =\nabla\phi_i(t, x)$ at a singular point
  $(t, x)$. %
  The coefficients $\pi_i$ correspond to probabilities that a
  trajectory $\gamma^\epsilon$ visits each of the domains where $\phi
  = \phi_j$. %
  Let us call a velocity $v^\dagger$ satisfying
  condition~\eqref{eq:46_1} \emph{self-consistent}. %

  The self-consistent velocity is a convex combination of
  \emph{velocities} seen by an infinitesimal observer leaving $(t, x)$
  with velocity~$v^\dagger$. %
  Compare this with the definition of admissible momentum $p^*$, which
  is a convex combination of \emph{momenta} see by a similar observer
  moving with velocity~$v^*$. %
  When $H(t, x, p) = |p|^2/2$ and $v = p$, self-consistent velocities
  and admissible velocities coincide. %
  It is however clear that in the case of a general nonlinear Legendre
  transform $v^\dagger \neq v^* = \nabla_p H(t, x, p^*)$. %

  To see this let us consider the case of the shock manifold of
  co-dimension one. %
  Then at every point of the shock manifold there are exactly two
  extreme values of the momenta $p_1$ and $p_2$, and a
  sub-differential is a straight interval $I$ connecting them. %
  Under the viscous regularization, the ``effective,'' or admissible,
  momentum $p^*$ is a point inside $I$. %
  Hence the ``effective,'' i.e., admissible velocity $v^*$ belongs to
  the Legendre image of $I$, which is the curve $J(s)=\{\nabla_p
  H(t,x,sp_1+(1-s)p_2)\colon 0 \leq s \leq 1\}$. %
  On the contrary, under the weak noise regularization the effective
  velocity $v^\dagger$ belongs to a straight interval $\bar J$
  connecting the extreme velocities $v_1=\nabla_pH(t,x,p_1)$ and
  $v_2=\nabla_pH(t,x,p_2)$. %
  It is easy to see that when $d \geq 3$ generally the curves $J$ and
  $\bar J$ do not intersect outside of the end-points, hence the
  effective velocities $v^*$ and $v^\dagger$ are necessarily
  different. %
  Of course in the one-dimensional case $J$ and $\bar J$ coincide. %
  The only other case when $J = \bar J$ is when the Legendre transform
  is a linear map, or equivalently when the Hamiltonian $H$ is
  quadratic in the momentum variable. %
  In both cases uniqueness of generalized characteristics is a
  well-established fact \cite{Brezis.H:1973,Dafermos.C:2005,Cannarsa.P:2004,Bogaevsky.I:2004,Bogaevsky.I:2006,Cannarsa.P:2009}. %

  Let $(t, x)$, $x \in \mathbf{R}^d$, be a singular point where
  $\partial\phi(t, x)$ is a $d$-dimensional simplex. %
  Obviously $v^\dagger$ is uniquely defined in the one-dimensional
  case. %
  Using straightforward but somewhat cumbersome analysis of all
  particular cases, one can check that in two dimensions there is a
  unique self-consistent value of velocity $v^\dagger(t, x)$, too. %
  However when $d \ge 3$, it is possible to construct a convex
  Hamiltonian and a piecewise linear viscosity solution~$\phi$ such
  that at a certain singular point $(t, x)$ there are three
  self-consistent values of velocity. %

  Observe that momenta form a vector space dual to that of velocities;
  fix a basis in the velocity space and use the standard scalar
  product to map momenta and velocities to the same
  space~$\mathbf{R}^3$. %
  We shall denote the coordinates by $x$, $y$, $z$ and write vectors
  as $a = (a_x, a_y, a_z)^{\mathsf T}$. %

  Perform a translation of $t$ and $x$ variables such that $t_0 = 0$,
  $x_0 = 0$ and choose momenta to be represented by vertices of a
  regular tetrahedron centered at the origin and symmetric with
  respect to the $xy$ and $xz$ planes:
  \begin{equation}
    \label{eq:39_1}
    p_1 = \begin{pmatrix} 1 \\ 0 \\ 1 \end{pmatrix},\
    p_2 = \begin{pmatrix} 1 \\ 0 \\ -1 \end{pmatrix},\
    p_3 = \begin{pmatrix} -1 \\ 1 \\ 0 \end{pmatrix},\
    p_4 = \begin{pmatrix} -1 \\ -1 \\ 0 \end{pmatrix}.
  \end{equation}
  Denote by $H_i$, $1\le i\le 4$, the corresponding values of the
  Hamiltonian (to be fixed later). %
  This set of momenta corresponds to the solution
  \begin{equation}
    \label{eq:40_1}
    \phi(t, x) = \min(x + z - tH_1, x - z - t H_2, y - x - t H_3,
    -x - y - t H_4)
  \end{equation}
  to the Hamilton--Jacobi equation~\eqref{eq:1} as $t \ge 0$. %

  If a set of velocities $\{v_i\colon 1\le i\le 4\}$, is the Legendre
  image of this set of momenta, they must satisfy conditions $H(p_i) +
  v_i\cdot(p_j - p_i) < H(p_j)$, or equivalently
  \begin{equation}
    \label{eq:41_1}
    -v_i\cdot p_i + H_i < -v_i\cdot p_j + H_j
  \end{equation}
  for all $i\neq j$. %
  Conversely, for any set of vectors $v_i$ and numbers~$H_i$, $1\le
  i\le 4$, that satisfy these conditions there exists a strictly
  convex Hamiltonian $H(p)$ such that $H(p_i) = H_i$ and $\nabla_p
  H(p_i) = v_i$. %
  To see this, define $\hat H(p) = \max\{H_i + v_i\cdot(p - p_i)\colon
  1\le i\le 4\}$, which satisfies $H(p_i) = H_i$ because
  of~\eqref{eq:41_1}, and ``smooth out'' the function~$\hat H$ to get
  strict convexity in such a way that the values at~$p_i$ are
  preserved. %

  Now take $H_i = |p_i|^2/2$ (velocities $v_i = \nabla_p H(p_i)$ are
  not yet fixed, so the Hamiltonian need not, and won't, coincide with
  $H(p) = |p|^2/2$ everywhere). %
  Conditions~\eqref{eq:41_1} are equivalent to the requirement that for
  each $1\le i\le 4$, the momentum vector closest to~$v_i$ is~$p_i$:
  indeed, adding $\frac 12|v_i|^2$ to both sides of~\eqref{eq:41_1}, we
  get $\frac 12|v_i - p_i|^2 < \frac 12|v_i - p_j|^2$ for all $j \neq
  i$. %
  In other words, each $v_i$ belongs to the Voronoi cell of~$p_i$. %

  Let now velocities be given by
  \begin{equation}
    \label{eq:42_1}
    v_1 = \begin{pmatrix} -\frac 12 \\ 0 \\ 2 \end{pmatrix},\
    v_2 = \begin{pmatrix} -\frac 12 \\ 0 \\ -2 \end{pmatrix},\
    v_3 = \begin{pmatrix} \frac 12 \\ 2 \\ 0 \end{pmatrix},\
    v_4 = \begin{pmatrix} \frac 12 \\ -2 \\ 0 \end{pmatrix}.
  \end{equation}
  All these velocities are in Voronoi cells of momenta with
  corresponding indices, which ensures that $v_i = \nabla_p H(p_i)$
  for a suitable convex Hamiltonian. %
  Consider now $v' = \frac 12 v_1 + \frac 12 v_2 = (-\frac 12, 0,
  0)^{\mathsf T}$. %
  It is easy to check using~\eqref{eq:40_1} for~$\Phi(1, v)$ that $I(v')
  = \{1, 2\}$, so $v'$ is self-consistent. %
  But symmetry implies that $v'' = \frac 12 v_3 + \frac 12 v_4$ is
  also self-consistent. %
  Moreover, the arithmetic average $v''' = 0$ of all $v_i$ is clearly
  self-consistent as well, which gives three distinct self-consistent
  values of velocity at the same point. %

  \section{Conclusions}
  \label{sec:conclusions}

  We conclude with a list of a few open problems concerning the
  approaches presented above, mostly the viscous regularization. %
  Of course the most important of these problems is the issue of
  uniqueness of the limit trajectories, discussed in
  Subsection~\ref{sec:uniq-probl-limit} above. %

  Furthermore, the flow of limit trajectories~$\gamma$, seen as a
  family of continuous maps of variational origin from initial
  coordinates $y = \gamma(0)$ to current coordinates $x = \gamma(t)$,
  is clearly relevant for optimal transportation problems
  \cite{Gangbo.W:1996,Villani.C:2009}. %
  An interesting problem suggested by B.~Khesin is to study the
  extremal properties of this flow. %
  Indeed it is known from \cite{Khesin.B:2007} that before the first
  shock formation the flow $\gamma_{ y}$ is an action minimizing flow
  of diffeomorphisms, while the first shock formation time $t^*$ marks
  a conjugate point in the corresponding variational problem. %
  According to the suggested view, the flow constructed above may be
  seen as a kind of saddle-point, rather than minimum, for a suitable
  transport optimization problem. %

  Finally, we have seen in Section~\ref{sec:regul-with-weak} that in
  dimensions $d \ge 3$ the self-consistent velocity, which plays the
  same role for a flow regularized with weak noise that the admissible
  velocity does for the viscosity regularization, may fail to be
  defined uniquely. %
  It is an interesting problem nevertheless to see whether a unique
  limiting flow still exists in the limit of weak noise in spite of
  nonuniqueness of admissible velocity.
  This problem carries a certain similarity with the problem of limit
  behaviour for one-dimensional Gibbs measures in the zero-temperature
  limit, in the case of nonunique ground states.

\begin{acknowledgements}
Some results of this paper, including the proof of
  Theorem~1, appeared in conference proceedings
  \cite{Khanin.K:2010a}. %
In the course of this work we benefitted from valuable
remarks of J{\'e}r{\'e}mie Bec, Patrick Bernard, Ilya
Bogaevsky, Yann Brenier, Philippe Choquard, Michael Dabkowski, Uriel
Frisch, Boris Khesin, and Thomas Str{\"o}mberg. %
It is a pleasure to recognize their help as well as the unique
environment of the Observatoire de la C{\^o}te d'Azur, where this work
was started and advanced. %
\end{acknowledgements}

\bibliographystyle{spmpsci}
\bibliography{dynamicsHJE}   

\end{document}